\documentclass[12pt,twoside]{amsart}
\usepackage{amsmath}
\usepackage{amsthm}
\usepackage{amsfonts}
\usepackage{amssymb}
\usepackage{latexsym}
\usepackage{mathrsfs}
\usepackage{amsmath}
\usepackage{amsthm}
\usepackage{amsfonts}
\usepackage{amssymb}
\usepackage{latexsym}
\usepackage{geometry}
\usepackage{dsfont}
\usepackage[dvips]{graphicx}
\usepackage{color}
\usepackage[all]{xy}

\date{}
\allowdisplaybreaks[4] \footskip=15pt
\renewcommand{\uppercasenonmath}[1]{}

\usepackage{graphicx,amssymb}
\usepackage[all]{xy}
\usepackage{amsmath}

\allowdisplaybreaks
\usepackage{amsthm}
\usepackage{color}

\theoremstyle{plain}
\newtheorem{theorem}{Theorem}[section]
\newtheorem{proposition}[theorem]{Proposition}
\newtheorem{lemma}[theorem]{Lemma}

\newtheorem{example}[theorem]{Example}

\newtheorem{definition}[theorem]{Definition}

\theoremstyle{definition}
\newtheorem*{acknowledgement}{Acknowledgement}

\theoremstyle{remark}
\newtheorem{remark}[theorem]{Remark}
\newcommand{\C}{\mathcal{C}}

\newcommand{\Tor}{\mbox{\rm Tor}}

\newcommand{\A}{\mathcal{A}}

\newcommand{\prodi}{\prod_{i\in \Lambda}}

\newcommand{\Id}{\mathrm{Id}}

\def\Hom{{\rm Hom}}
\def\Ext{{\rm Ext}}
\def\Tor{{\rm Tor}}

\def\Ker{{\rm Ker}}

\def\Im{{\rm Im}}
\def\Coker{{\rm Coker}}
\def\End{{\rm End}}

\def\Id{{\rm Id}}

\begin{document}
\begin{center}
{\large  \bf Some remarks on  $u$-$S$-Noetherian and $u$-$S$-coherent rings}

\vspace{0.5cm} Xiaolei Zhang, Wei Qi
\bigskip

{\footnotesize
	School of Mathematics and Statistics,  Tianshui Normal University,   741001 Tianshui, China\\
}
\end{center}

\bigskip
\centerline { \bf  Abstract}
\bigskip
\leftskip10truemm \rightskip10truemm \noindent

 In this paper, we give some new characterizations of $u$-$S$-Noetherian rings  and $u$-$S$-coherent rings in terms of uniform $S$-versions of  injective precovers, flat preenvelopes  and absolutely pure modules. Moreover, we give a negative answer to a question proposed by Bouziri.
\vbox to 0.3cm{}

\leftskip0truemm \rightskip0truemm
\bigskip

\section{Introduction}
Throughout this paper, $R$ is always  a commutative ring with an identity and $S$ is always a multiplicative subset of $R$, that is, $1\in S$ and $s_1s_2\in S$ for any $s_1\in S$ and any $s_2\in S$. Let $M$ be an $R$-module. Denote by $M^+=\Hom_{\mathbb{Z}}(M,\mathbb{Q}/\mathbb{Z})$.

In 2002, Anderson \cite{ad02} introduced the notions of $S$-finite modules and $S$-Noetherian rings which can be seen as $S$-versions of finitely generated modules and Noetherian rings, respectively.  An $R$-module $M$ is said to be \emph{$S$-finite} (with respect to $s$) provided that there is a finitely generated submodule $F$ of $M$ such that $sM\subseteq F$ for some $s\in S$;  and a ring $R$ is called an \emph{$S$-Noetherian ring} if every ideal of $R$ is $S$-finite.  In 2018, Bennis and Hajoui \cite{bh18} introduced the notions of $S$-finitely presented modules and $S$-coherent rings which can be seen as $S$-versions of finitely presented modules and coherent rings, respectively. An  $R$-module $M$ is called $S$-finitely presented
provided that there exists an exact sequence of $R$-modules $0\rightarrow K\rightarrow F\rightarrow M\rightarrow 0$ with $K$ $S$-finite and $F$ finitely generated free. A ring $R$ is called an \emph{$S$-coherent ring} if $R$ itself is an $S$-coherent $R$-module, that is, every finitely generated ideal of $R$ is $S$-finitely presented.

However, the selected $s$ in the definitions of $S$-Noetherian rings and $S$-coherent rings is not uniform on ideals. So Chen et al. \cite{CKQWZ} introduced the notion of $u$-$S$-Noetherian rings. A ring $R$ is called a {\it $u$-$S$-Noetherian ring} (abbreviates \emph{uniformly  $S$-Noetherian ring}) provided that there exists  $s\in S$ such that  any ideal $I$ of $R$ is  $S$-finite with respect to $s$. Subsequently, Zhang \cite{Zuscoh-24} introduced the notion of uniformly $S$-coherent rings.  An  $R$-module $M$ is said to be
\emph{$u$-$S$-finitely presented} (abbreviates \emph{uniformly  $S$-finitely presented}) (with respect to $s$) provided that there is an exact sequence $$0\rightarrow T_1\rightarrow F\xrightarrow{f} M\rightarrow T_2\rightarrow 0$$ with $F$ finitely presented and $sT_1=sT_2=0$. A ring  $R$ is called a
\emph{$u$-$S$-coherent ring} (abbreviates uniformly  $S$-coherent ring) (with respect to $s$) provided that there exists $s\in S$ such that  any finitely generated ideal of $R$ is $u$-$S$-finitely presented with respect to $s$.  It was proved in \cite[Proposition 3.10]{Zuscoh-24} that every $u$-$S$-Noetherian ring is $u$-$S$-coherent.

It is well-known that a ring $R$ is Noetherian if and only if any direct sum  of injective $R$-modules is injective,  if and only if every $R$-module has a injective precover. A ring $R$ is coherent if and only if  any pure quotient module of an absolutely pure module is  absolutely pure, if and only if any direct limit of absolutely pure modules is absolutely pure, if and only if the class of flat modules is preenveloping. The main aim of this paper is to give uniform $S$-versions of several classical results.

Since the paper involves uniformly torsion theory, we give a quick review (see  \cite{z21} for more details). An $R$-module $T$ is called  $u$-$S$-torsion (with respect to $s$) provided that there exists  $s\in S$ such that $sT=0$. 
An $R$-sequence $ A\xrightarrow{f} B\xrightarrow{g} C$ is said to be \emph{$u$-$S$-exact at $B$} (with respect to $s$) if $s\Ker(g)\subseteq \Im(f)$ and $s\Im(f)\subseteq \Ker(g)$ for some $s\in S$. An $R$-sequence $\cdots\rightarrow A_{n-1}\xrightarrow{f_n} A_{n}\xrightarrow{f_{n+1}} A_{n+1}\rightarrow\cdots$ is called a \emph{$u$-$S$-exact sequence} if it is $u$-$S$-exact at each $A_n$. An $R$-homomorphism $f:M\rightarrow N$ is a \emph{$u$-$S$-monomorphism} $($resp.,   \emph{$u$-$S$-epimorphism}, \emph{$u$-$S$-isomorphism}$)$  (with respect to $s$) provided $0\rightarrow M\xrightarrow{f} N$   $($resp., $M\xrightarrow{f} N\rightarrow 0$, $0\rightarrow M\xrightarrow{f} N\rightarrow 0$ $)$ is  $u$-$S$-exact  (with respect to $s$).
Let $M$ and $N$ be $R$-modules. We say that $M$ is $u$-$S$-isomorphic to $N$ if there exists a $u$-$S$-isomorphism $f:M\rightarrow N$. If so, it was proved in \cite[Lemma 1.2]{ZQusproj} that there is $s\in S$  and a $u$-$S$-isomorphism $f':N\rightarrow M$ such that $f\circ f' = s \Id_N$ and $f' \circ f = s \Id_{M}.$

\section{Characterizing $u$-$S$-coherent rings by $u$-$S$-flat modules}

Let $R$ be a ring, $S$ be a multiplicative subset of $R$ and $s\in S$.
Recall from \cite{z21} that an $R$-module $F$ is said to be \emph{$u$-$S$-flat} (abbreviates \emph{uniformly $S$-flat}) (with respect to $s$) provided that for any  exact sequence $0\rightarrow A\rightarrow B\rightarrow C\rightarrow 0$, the induced sequence $0\rightarrow A\otimes_RF\rightarrow B\otimes_RF\rightarrow C\otimes_RF\rightarrow 0$ is  $u$-$S$-exact (with respect to $s$). It follows by \cite[Proposition 4.3]{Zuscoh-24} that an $R$-module $F$ is $u$-$S$-flat if and only if  there exists  $s\in S$ such that $\Tor_1^R(R/I,F)$ is $u$-$S$-torsion with respect to $s$ for any ideal $I$ of $R$ provided that $S$ is  regular, i.e., every element in $S$ is a nonzero-divisor.

It  follows from \cite[Theorem 4.4]{Zuscoh-24} that  a ring $R$   is a $u$-$S$-coherent ring if and only if   there is $s\in S$ such that any direct product of flat modules (projective modules, copies of $R$) is $u$-$S$-flat with respect to $s$. It is well-known that a ring $R$ is coherent if and only if every $R$-module has a flat preenvelope. Now, we will try to give the uniform $S$-version of this result.

First, we introduce a uniform $S$-version of the (pre)enveloping properties for a class of modules.
\begin{definition} Let $M$ be an $R$-module and $\C$ a class of $R$-modules.
	\begin{enumerate}
		\item An $R$-homomorphism $f\in \Hom_R(M,C)$ with $C\in\C$ is called a  $\C$  $u$-$S$-preenvelope of $M$ $($with respect to $s)$ provided that, for any $C'\in \C$,  $\Hom_R(f,C'):\Hom_R(C,C')\rightarrow \Hom_R(M,C')$ is a $u$-$S$-epimorphism  $($with respect to $s)$.
		
		\item A $\C$ $u$-$S$-preenvelope $f\in \Hom_R(M,C)$ of $M$ is called a $\C$ $u$-$S$-envelope $($with respect to $s)$ of $M$, provided that, for any $\alpha\in \End_R(C)$, $f=\alpha\circ f $ implies that $\alpha$ is a $u$-$S$-isomorphism $($with respect to $s)$.
		
		\item If every $R$-module has a  $\C$  $u$-$S$-preenvelope (resp., $u$-$S$-envelope), then $\C$ is called a $u$-$S$-preenveloping (resp., $u$-$S$-enveloping) class.
	\end{enumerate}
\end{definition}

Recall from \cite{ZQusproj} that  a short $u$-$S$-exact sequence  $\xi: 0\rightarrow A\xrightarrow{f} B\xrightarrow{g} C\rightarrow 0$   is said to be \emph{ $u$-$S$-split} (with respect to $s$) provided that there is  $s\in S$ and $R$-homomorphism $f':B\rightarrow A$ such that $f'(f(a))=sa$ for any $a\in A$, that is, $f'\circ f=s\Id_A$.

\begin{lemma}\label{ussplit-usflat}
	Let $\xi:0\rightarrow A\rightarrow B\rightarrow C\rightarrow 0$ be a $u$-$S$-split  short $u$-$S$-exact sequence. Suppose $B$ is $u$-$S$-flat. Then so is $A$ and $C$.
\end{lemma}
\begin{proof} 	Let $\xi:0\rightarrow A\rightarrow B\rightarrow C\rightarrow 0$ be a $u$-$S$-split  short $u$-$S$-exact sequence  with $B$  $u$-$S$-flat. It follows by \cite[Proposition 2.3]{Z23usapm} that $\xi$ is $u$-$S$-pure. Hence $C$ is $u$-$S$-flat by   \cite[Proposition 2.5]{Z23usapm}. By the proof of \cite[Lemma 2.4]{ZQusproj}, there is  a $u$-$S$-split  short $u$-$S$-exact sequence
$\xi':0\rightarrow C\rightarrow B\rightarrow A\rightarrow 0$ induced by $\xi$.	Hence $A$ is also $u$-$S$-flat.	
\end{proof}

 For any $s\in S$, let $s\mathcal{F}$ denote the class of all $R$-modules that are $u$-$S$-flat with respect to $s$. It was proved in \cite{Bourzri26} that if there exists $s\in S$ such that every $R$-module has an $s\mathcal{F}$-preenvelope, then $R$ is $u$-$S$-coherent. However, the converse is still  open (see \cite[Question 2.12]{Bourzri26}). Next, we will give a new characterization of $u$-$S$-coherent rings in terms of uniform $S$-version of preenvelopes under a mild condition.

 \begin{theorem}\label{usflatpreenvelope}
 	Let $R$ be a ring, $S$ be a regular  multiplicative subset of $R$. Consider the following statements:
 	\begin{enumerate}
 		\item $R$ is a $u$-$S$-coherent ring;
 		\item for all $s\in S$, any direct product of $R$-modules in   $s\mathcal{F}$ is $u$-$S$-flat;
 		
 	\item   any direct product of flat $R$-modules $($projective $R$-modules, copies of $R)$ is $u$-$S$-flat;
 		\item   any $R$-module has a $u$-$S$-flat $u$-$S$-preenvelope.	
 	\end{enumerate}
 Then $(1)\Leftrightarrow (2)\Leftrightarrow (3)\Leftarrow (4)$. If, moreover, every $u$-$S$-flat module is $u$-$S$-isomorphic  to a flat module, then $(3)\Rightarrow (4)$.
 \end{theorem}
 \begin{proof}
 	
$(2)\Rightarrow  (3)$ Trivial.
 	
 $(1)\Rightarrow (2)$ Suppose $R$ is $u$-$S$-coherent with respect to $s'\in S$.    Let $\{F_i\mid i\in \Lambda\}$ be a family of $u$-$S$-flat $R$-modules with respect to $s$ and $I$ a finitely generated ideal of $R$. Then $I$ is $u$-$S$-finitely presented with respect to $s'$. So we have an exact sequence $0\rightarrow T'\xrightarrow{t'} K\xrightarrow{f} I\xrightarrow{t} T\rightarrow 0$ with $K$ finitely presented and $s'T=s'T'=0$. Set $\Im(f)=K'$, $f_1:K\rightarrow K'$ and $f_2:K'\rightarrow I$ to be the natural homomorphisms.
 	
 	 Consider the following commutative diagram with exact rows:
 	$$\xymatrix{
 		& T'\otimes_R \prodi F_i\ar[d]_{\alpha}\ar[r]^{t'\otimes \Id_{\prod F_i}} & K\otimes_R  \prodi F_i \ar[d]_{\gamma}^{\cong}\ar[r]^{f_1\otimes \Id_{\prod F_i}} & K'\otimes_R  \prodi F_i \ar[d]^{\beta}\ar[r]^{} &  0\\
 X \ar@{>->}[r]^{} & \prodi (T'\otimes_R F_i )\ar[r]^{\prod t'\otimes \Id_{ F_i}} &\prodi (K\otimes_R F_i ) \ar[r]^{\prod f_1\otimes \Id_{ F_i}} & \prodi (K'\otimes_R F_i) \ar[r]^{} &  0,\\}$$
	By  \cite[Theorem 2.6.11]{FK24}, $\gamma$ is an isomorphism.
 Since each $T'\otimes_R F_i$ is $u$-$S$-torsion with respect to $s'$, $\prodi (T'\otimes_R F_i )$ and hence $X$ is also $u$-$S$-torsion with respect to $s'$. We claim that $\Ker(\beta)$ is a $u$-$S$-torsion with respect to $s'$. Indeed, let $x\in \Ker(\beta)$. Then there is $y\in K\otimes_R  \prodi F_i$ such that $(\prod f_1\otimes \Id_{ F_i})\circ\gamma (y)=0,$ and $f_1\otimes \Id_{\prod  F_i} (y)=x.$ Hence, there is $z\in \Im(\prod t'\otimes \Id_{F_i})$ such that $\gamma(y)=z$. So $s'x=s'f_1\otimes \Id_{\prod  F_i} (y)=f_1\otimes\Id_{\prod  F_i} (s'y)=$ $f_1\otimes\Id_{\prod  F_i} \gamma^{-1}\gamma(s'y)=f_1\otimes\Id_{\prod  F_i} \gamma^{-1}s'\gamma(y)=f_1\otimes\Id_{\prod  F_i} \gamma^{-1}(s'z)=f_1\otimes\Id_{\prod  F_i} \gamma^{-1}(0)=0$. The claim holds.

Consider the following commutative diagram:
 $$\xymatrix{
 	& K'\otimes_R \prodi F_i\ar@{->>}[d]_{\beta}\ar[r]^{f_2\otimes \Id_{\prod F_i}} & I\otimes_R  \prodi F_i \ar[d]^{\theta}\ar[r]^{t\otimes \Id_{\prod F_i}} & T\otimes_R  \prodi F_i \ar[d]^{\delta}\ar[r]^{} &  0\\
\prodi  X_i \ar@{>->}[r]^{} & \prodi (K'\otimes_R F_i )\ar[r]^{\prod f_2\otimes \Id_{F_i}} &\prodi (I\otimes_R F_i ) \ar[r]^{\prod t\otimes \Id_{F_i}} & \prodi (T\otimes_R F_i) \ar[r]^{} &  0,\\}$$	
where $X_i:=\Ker(f_2\otimes \Id_{F_i}).$ We claim that $\Ker(\theta)$ is also  $u$-$S$-torsion  with respect to $s(s')^2$. Indeed,  let $a\in\Ker(\theta)$. Since $T\otimes_R  \prodi F_i$ is $u$-$S$-torsion  with respect to $s'$, we have $t\otimes \Id_{\prod F_i}(s'a)=s'(t\otimes \Id_{\prod F_i}(a))=0$. So there $b\in K'\otimes_R \prodi F_i$ such that $f_2\otimes \Id_{\prod F_i}(b)=s'a.$ Then $(\prod f_2\otimes \Id_{F_i})\circ\beta(b)=\theta\circ(f_2\otimes \Id_{\prod F_i})(b)=\theta(s'a)=s'\theta(a)=0$.
  Since each $F_i$ is $u$-$S$-flat with respect to $s$, we have $\prodi X_i$ is $u$-$S$-torsion with respect to $s$. Hence $\beta(sb)=s\beta(b)=0$. Since $\Ker(\beta)$ is  $u$-$S$-torsion with respect to $s'$, we have $ss'b=0$. Consequently, $s(s')^2a=ss'f_2\otimes \Id_{\prod F_i}(b)= f_2\otimes \Id_{\prod F_i}(ss'b)=0.$
 	
 	Now we consider the  following commutative diagram with exact rows:
 	$$\xymatrix{
 		0\ar[r]& \Tor_1^R(R/I,\prodi F_i)\ar[d]_{}\ar[r]^{} & I\otimes_R \prodi F_i \ar[d]^{\theta}\ar[r]^{} & R\otimes_R  \prodi F_i \ar[d]^{}\\
 		 & \prodi\Tor_1^R(R/I, F_i)\ar[r]^{} &\prodi (I\otimes_R F_i ) \ar[r]^{} & \prodi (R\otimes_R F_i),\\}$$
Since each $F_i$ is $u$-$S$-flat with respect to $s$, $\prodi\Tor_1^R(R/I, F_i)$ is $u$-$S$-torsion  with respect to $s$. Hence $s\Tor_1^R(R/I,\prodi F_i)\subseteq \Ker(\theta)$. So $\Tor_1^R(R/I,\prodi F_i)$ is $u$-$S$-torsion. Hence  $\prodi F_i$ is $u$-$S$-flat by  Proposition \cite[Theorem 4.3]{Zuscoh-24}.
 	
 $(4)\Rightarrow (1)$ Let $\mathscr{I}=\{I_i\mid i\in\Gamma\}$ be the set of all finitely generated ideals of $R$, and let $0\rightarrow L_{I_i}\rightarrow F_i\rightarrow I_i\rightarrow 0$ be an exact sequence with $F_i$ finitely generated free for each $i\in\Gamma$. Let $\Lambda$  be cardinal greater than the sum of all ordinals of $L_{I_i}$s.
 		
 For each $I\in \mathscr{I}$, 		
consider the  following commutative diagram with exact rows:
 	$$\xymatrix{
 		&I\otimes_R \prodi R\ar[d]_{g}\ar[r]^{f} & R\otimes_R  \prodi R \ar[d]_{}^{\cong}\ar[r]^{} & R/I\otimes_R  \prodi R \ar[d]_{}^{\cong}\ar[r]^{} &  0\\
 		0 \ar[r]^{} & \prodi (I\otimes_RR)\ar[r]^{} &\prod_{i\in \Lambda} (R\otimes_RR ) \ar[r]^{} & \prodi (R/I\otimes_RR ) \ar[r]^{} &  0.\\}$$
By assumption, $\prodi R$ is a $u$-$S$-flat module with respect to some fixed element $s\in S$, it follows that $f$ is a $u$-$S$-monomorphism with respect to $s$. So $g$ is also a $u$-$S$-monomorphism with respect to $s$. 	Consider the  following commutative diagram with exact rows:
 	$$\xymatrix{
 		&L_I\otimes_R \prodi R\ar[d]_{h}\ar[r]^{} & F\otimes_R  \prodi R \ar[d]_{}^{\cong}\ar[r]^{} & I\otimes_R  \prodi R \ar[d]_{}^{g}\ar[r]^{} &  0\\
 		0 \ar[r]^{} & \prodi (L_I\otimes_RR)\ar[r]^{} &\prod_{i\in \Lambda} (F\otimes_RR ) \ar[r]^{} & \prodi (I\otimes_RR ) \ar[r]^{} &  0.\\}$$
 Since $g$ is a  $u$-$S$-monomorphism with respect to $s$, $h$ is a $u$-$S$-epimorphism with respect to $s$.  We will show $L_I$ is $S$-finite with respect to $s$. Indeed, consider the following exact sequence
 	$$\xymatrix{
 		L_I\otimes_R  R^\Lambda \ar[rr]^{h}\ar@{->>}[rd] &&L_I^\Lambda \ar[r]^{} & T\ar[r]^{} &  0\\
 		&\Im h \ar@{^{(}->}[ru] &&  &   \\}$$
 	with $T$ a $u$-$S$-torsion module with respect to $s$.
 	Consider an element $x \in \prod_{i \in \Lambda} L_{I_i}$ such that every element of $L_{I}$ appears at least once among the components of $x$. Then  $sx\subseteq \Im h$.
 	Subsequently, there exist $m_j\in L_I, r_{j,i}\in R, i\in \Lambda, j=1,\dots,n$
 	such that
 	$$sx=h(\sum_{j=1}^n m_j\otimes (r_{j,i})_{i\in L_I})=(\sum_{j=1}^n m_j r_{j,i})_{i\in L_I}.$$
 	Set $U=\langle m_j\mid j=1,\dots,n\rangle $ be the finitely generated submodule of $L_I$. Now, for any $m\in L_I$, $sm\in \langle \sum_{j=1}^n m_j r_{j,m} \rangle\subseteq  U$, thus the embedding map $U\hookrightarrow L_I$ is a $u$-$S$-isomorphism with respect to $s$ and so $L_I$ is  $S$-finite  with respect to $s$. Consequently, $I$ is $u$-$S$-finitely presented with respect to $s$. Hence,  $R$ is $u$-$S$-coherent with respect to $s$.

 (4)$\Rightarrow$ (3): Let 	$\{F_i\mid i\in\Lambda\}$ be a family of flat modules. Then, by assumption, $\prodi F_i$ has a $u$-$S$-flat $u$-$S$-preenvelope  $f:\prodi F_i\rightarrow F$ with $F$ $u$-$S$-flat. Hence there exists $s\in S$ such that for any $i\in\Lambda$ we have the following commutative diagram:
 $$\xymatrix{
 	 \prodi F_i\ar[r]^{f}\ar[d]^{\pi_i} & F \ar[d]^{g_i} \\
  F_i\ar[r]^{\times s} &F_i,}$$
 Consequently, $s\Id_{\prodi F_i}=(\prodi g_i)\circ f$. So  $$0\rightarrow \prodi F_i\xrightarrow {f}F\rightarrow \Coker (f)\rightarrow 0$$	is a $u$-$S$-split  short $u$-$S$-exact sequence.  It follows by Lemma \ref{ussplit-usflat} that $\prodi F_i$ is $u$-$S$-flat.

Now, we assume that  every $u$-$S$-flat module is $u$-$S$-isomorphic to a flat module.
 	
 (3)$\Rightarrow$ (4): 	Let $M$ be an $R$-module and Card$M\leq \aleph_\beta$. Then by \cite[Lemma  5.3.12]{EJ11} there is an infinite cardinal $\aleph_\alpha$ such that if $F$ is a flat module and $S$ is a submodule of $F$ with Card$S\leq \aleph_\alpha$, there is a pure, hence flat, submodule $G$ of $F$ with $S\subseteq G$ and  Card$G\leq \aleph_\alpha$. So for any flat $R$-module $F_i$ and any homomorphism $\phi_i': M\rightarrow F_i$, there is a   flat  $R$-submodule $G_i$ of $F_i$ satisfying that $G_i$ contains $f(M)$  and  Card$G_i\leq \aleph_\alpha$. So,  we have a natural factorization $M\xrightarrow{\phi_i} G_i\rightarrow F_i$.

Note Card$G_i\leq \aleph_\alpha$ for each $i$, so all such $G_i$ (up to isomorphism) constitute a set, which is denoted by $\{G_i\mid i\in \Lambda\}$.  We claim that the natural $R$-homomorphism $\phi:=\prodi \phi_i:M\rightarrow \prodi G_i$ is $u$-$S$-flat $u$-$S$-preenvelope of $M$. Indeed, by assumption, $\prodi G_i$ is $u$-$S$-flat. Let $f:M\rightarrow F$ be an $R$-homomorphism with $F$ $u$-$S$-flat. Then $F$ is $u$-$S$-isomorphic to a flat module $F'$.  Hence there is $s\in S$  and  $u$-$S$-isomorphisms $t:F'\rightarrow F$ and $u$-$S$-isomorphism $t':F\rightarrow F'$ such that $t \circ t' = s \Id_F$ and $t' \circ t = s \Id_{F'}$ by \cite[Lemma 1.2]{ZQusproj}. By the property of each $G_i,$
   there exists an $R$-homomorphism $g:\prodi G_i\rightarrow F'$ such that $g\circ \phi=t'\circ f$:
 $$\xymatrix{
 	M\ar[r]^{\phi}\ar[d]_{f} & \prodi G_i \ar[d]^{g}\\
 	F\ar@/^/[r]^{t'} &F'\ar@/^/[l]^{t}.}$$
 Hence, $(t\circ g)\circ \phi=t\circ t'\circ f=sf$. Consequently, $\phi$ is a $u$-$S$-flat $u$-$S$-preenvelope.	
 \end{proof}

\begin{remark} The equivalence of $(1)$ and $(3)$ in Theorem \ref{usflatpreenvelope}  is slightly different with that in \cite[Theorem 4.4]{Zuscoh-24}. As to the equivalence of $(1)$ and $(4)$, we don't know whether every $u$-$S$-flat module   is $u$-$S$-isomorphic  to a flat module in general.
\end{remark}

\section{Characterizing $u$-$S$-coherent rings by $u$-$S$-absolutely pure modules}


It is well-known that a ring $R$ is coherent if and only if any pure quotient of  an absolutely pure $R$-module is  absolutely pure, if and only if any direct limit of  absolutely pure $R$-modules is  absolutely pure (see \cite{S70}). In this section, we will give a uniform $S$-versions of this result.	

Recall the notion of uniformly $S$-absolutely pure modules from \cite{Z23usapm}. An $R$-module $E$ is said to be \emph{$u$-$S$-absolutely pure} (abbreviates \emph{uniformly $S$-absolutely pure}) (with respect to $s$) if
 there exists an element $s\in S$ satisfying that  for any finitely presented $R$-module $N$, $\Ext_R^1(N,E)$ is  $u$-$S$-torsion with respect to $s$, or equivalently,  there exists an element $s\in S$ satisfying that  if  $P$ is finitely generated projective, $K$ is a finitely generated submodule of $P$ and  $f:K\rightarrow E$ is an  $R$-homomorphism, then there is an $R$-homomorphism $g:P\rightarrow E$ such that $sf=gi$.

\begin{lemma}\label{uscohusfp}	Let $R$ be a ring, $S$ be a  multiplicative subset of $R$, and $M$ an $R$-module. Suppose  $M$  is  $u$-$S$-coherent with respect to some $s\in S$. Then $M$ is  $u$-$S$-finitely presented with respect to $s^2$.
\end{lemma}
\begin{proof}
Suppose  $M$  is  $u$-$S$-coherent with respect to some $s\in S$. Then $M$ is $S$-finite with respect to $s$, and hence there is  a finitely generated $R$-submodule $K$ of $M$ such that $sM\subseteq K\stackrel{i}{\hookrightarrow} M$. Note that $K$ is $u$-$S$-finitely presented with respect to $s$. So there is an exact sequence $0\rightarrow T_1\rightarrow F\xrightarrow{f} K\rightarrow T_2\rightarrow 0$ with $F$ finitely presented and $sT_1=sT_2=0$. Then one can construct an exact sequence $0\rightarrow T_1\rightarrow F\xrightarrow{f\circ i} M\rightarrow T_3\rightarrow 0$ with $s^2T_3=0$. Hence,  $M$ is  $u$-$S$-finitely presented with respect to $s^2$.
\end{proof}

Let $\xi: 0\rightarrow A\xrightarrow{f} B\xrightarrow{g} C\rightarrow 0$ be a $u$-$S$-short exact sequence. Then $\xi$ is said to be $u$-$S$-split (with respect to $s$) provided that there is  $s\in S$ and $R$-homomorphism $f':B\rightarrow A$ such that $f'(f(a))=sa$ for any $a\in A$, that is, $f'\circ f=s\Id_A$.

\begin{lemma}\label{dir usfp} Let $\xi: 0\rightarrow A\xrightarrow{f} B\xrightarrow{g} C\rightarrow 0$ be a $u$-$S$-split short exact sequence of $R$-modules with respect to $s$. Suppose $B$ is finitely presented, then $C$ is $u$-$S$-finitely presented  with respect to $s$.
\end{lemma}
\begin{proof} First, we claim that $A$  is $S$-finite with respect to $s$.  Indeed, let $f': B\rightarrow A$ be $R$-homomorphisms such that $f'\circ f=s\Id_A$. Then for any $a\in A$, we have $sa=f'(f(a))$. Hence, $sA\subseteq \Im(f')\subseteq A$. Since $B$ is finitely presented, we have $\Im(f')$ is finitely generated. Consequently, $A$  is $S$-finite with respect to $s$.

Next, we claim that $C$ is $u$-$S$-finitely presented  with respect to $s$. Indeed, Consider the following commutative:
	$$\xymatrix@R=20pt@C=25pt{
		0 \ar[r]^{}&\Im(f') \ar[r]^{}\ar@{^{(}->}[d]^{}&B \ar[r]^{}\ar@{=}[d]^{}&L\ar@{->>}[d]^{\pi}\ar[r]  &0\\
0 \ar[r]^{}&A\ar@{->>}[d]^{} \ar[r]^{f}&B \ar[r]^{g}&C\ar[r]  &0\\
		&T&& &\\}$$
 Then $sT=0.$ Since $\Im(f')$ is finitely generated and $B$ is finitely presented, $L$ is also finitely presented. It follows by the snake Lemma that $\Ker(\pi)\cong T$, and hence $s\Ker(\pi)=0$. Considering the exact sequence $0\rightarrow \Ker(\pi)\rightarrow L\rightarrow C\rightarrow 0$, we have $C$ is $u$-$S$-finitely presented  with respect to $s$.
\end{proof}

Recently, Bouziri \cite{Bourzri26} characterized $u$-$S$-coherent rings by dualising $u$-$S$-absolutely pure modules. Now, we give a new characterization of  $u$-$S$-coherent rings by considering the pure quotients or direct limits of  $u$-$S$-absolutely pure modules. Recall that a multiplicative subset $S$ of $R$ is said to be anti-Archimedean if $\bigcap\limits_{n\geq 1}s^nR\bigcap S\not=\emptyset.$

\begin{theorem}\label{s-d-d-non}
Let $R$ be a ring, $S$ be an anti-Archimedean  multiplicative subset of $R$. Then the following statements are equivalent:
	\begin{enumerate}
		\item  $R$ is $u$-$S$-coherent;
		\item there is $s\in S$ such that any pure quotient of an absolutely pure $R$-module is $u$-$S$-absolutely pure with respect to $s$;	
		\item  there is $s\in S$ such that  any direct limit of  absolutely pure $R$-modules is $u$-$S$-absolutely pure with respect to $s$.		
	\end{enumerate}
\end{theorem}
\begin{proof} $(1)\Rightarrow (2)$: Suppose  $R$ is a $u$-$S$-coherent ring with respect to some $s'\in S$. Let $B$ be an absolutely pure $R$-module. Let $A$ be a pure submodule of $B$, and let $C = B/A$. Let $s\in \bigcap\limits_{n\geq 1}(s')^nR\bigcap S$. We need to show that $C$ is $u$-$S$-absolutely pure with respect to  $s$. Let $P$ be a finitely generated projective $R$-module and $K$ a finitely generated submodule of $P$. Let $i:K\rightarrow P$ be the natural embedding map,  and  $f_C:K\rightarrow C$ an $R$-homomorphism. We claim that there is an $R$-homomorphism $g:P\rightarrow C$ such that $sf_C=gi$.   Consider the following exact sequence
	$0\rightarrow L\xrightarrow{i_L} P'\xrightarrow{\pi_{P'}} K\rightarrow 0$ with $P'$ finitely generated projective. Consider the following commutative diagram with exact sequences.
	$$\xymatrix@R=25pt@C=40pt{
		0\ar[r]^{}&L \ar[r]^{i_L}\ar[d]_{f_A} &P' \ar[d]^{f_B}\ar[r]^{\pi_{P'}} &\ar[d]^{f_C} K  \ar[r]^{}& 0 \\
		0\ar[r]&A\ar[r]^{i_A} &B\ar[r]^{\pi_B} & C \ar[r]^{}&0.\\ }$$
	Since $R$ is $u$-$S$-coherent with respect to $s'$, the finitely generated projective $R$-module $P$ is a $u$-$S$-coherent $R$-module with respect to some $t'\in S$, a power of $s'$, and so is its finitely generated  $R$-submodule $K$. It follows by \cite[Theorem 3.2(2)]{Zuscoh-24} that $L$  is also $u$-$S$-coherent with respect to $t'$,
	and so is $u$-$S$-finitely presented  with respect to $t=t'^2$ by Lemma \ref{uscohusfp}. Therefore there is an exact sequence $$0\rightarrow T_1\rightarrow K'\xrightarrow{i_{K'}} L\xrightarrow{\pi_L} T\rightarrow 0$$ with $tT=tT_1=0$    and $K'$ finitely presented. Consider the following push-out:
	
	$$\xymatrix@R=20pt@C=25pt{ & & 0\ar[d]&0\ar[d]&\\
		0 \ar[r]^{} & \Im(i_{K'})\ar@{=}[d]\ar[r]^{e_L} & L\ar[d]_{i_L}\ar[r]&T\ar[d]\ar[r]^{} &  0\\
		0 \ar[r]^{} & \Im(i_{K'})\ar[r]^{e_{P'}} & P' \ar[d]\ar[r]^{} &X\ar[d]^{\pi_K}\ar[r]^{} &  0\\
		&  & K \ar[d]\ar@{=}[r]^{} &K\ar[d] &  \\
		& & 0 &0 &\\}$$
It follows by \cite[Lemma 4]{M67} that $A$ is absolutely pure.  Consider the following commutative diagram:	$$\xymatrix@R=25pt@C=40pt{	0\ar[r]^{}&\Im(i_{K'}) \ar[r]^{e_{P'}}\ar[d]_{e_L} &P'\ar@{.>}[ldd]^{g_A} \ar@{=}[d]^{}\ar[r]^{} &\ar[d]^{\pi_K} X \ar@{.>}[ldd]^{g_B} \ar[r]^{}& 0 \\
 	0\ar[r]^{}&L \ar[r]^{i_L}\ar[d]_{f_A} &P' \ar[d]^{f_B}\ar[r]^{} &\ar[d]^{f_C} K  \ar[r]^{}& 0 \\
 	0\ar[r]&A\ar[r]^{i_A} &B\ar[r]^{\pi_B} & C \ar[r]^{}&0.\\ }$$
 Since $\Im(i_{K'})$ is finitely generated, then   there is an $R$-homomorphism $g_A:P'\rightarrow A$ such that $f_A\circ e_L=g_A\circ e_{P'}$ (see \cite[Corollary 2]{M70}).  And it follows by \cite[Exercise 1.60]{FK16} that there is an $R$-homomorphism $g_B:X\rightarrow B$ such that $\pi_B\circ g_B=f_C\circ \pi_K$.
 Since $K$ is $u$-$S$-isomorphic to $X$ with respect to $t$, by the proof of \cite[Proposition 1.1]{Z23usapm} there exist $\pi'_K:K\rightarrow X$ and $t_1\in S$, which is also a power of $s'$ say $t_1=s'^n$, such that $\pi_K\circ \pi'_K=t_1\Id_K$ and $\pi'_K\circ \pi_K=t_1\Id_X$.

Since $B$ is absolutely pure, it follows by \cite[Corollary 2]{M70} again that there exists $R$-homomorphism $g'_B:P\rightarrow B$ such that the following diagram is commutative:
	$$\xymatrix@R=30pt@C=40pt{
		K\ar[r]^{i}\ar[d]_{g_B\circ\pi_K'}&P\ar@{.>}[ld]^{g'_B}\\
		B&\\}$$
	We have $t_1f_C=f_C\circ\pi_K\circ\pi'_K=\pi_B\circ g_B\circ \pi'_K=\pi_B\circ g'_B\circ i=g'\circ i$, where $g':=\pi_B\circ g'_B$. Since $s\in \bigcap\limits_{n\geq 1}(s')^nR\bigcap S$, $s=s'^nr_n=t_1r_n\in S$ and $r_n\in R$. Let $g=r_ng'$. Then $sf_C=g\circ i$. It follows by \cite[Theorem 3.2]{Z23usapm} that $C$ is a $u$-$S$-absolutely pure $R$-module with respect to $s$.
	
	$(2)\Rightarrow (3)$:	Let $\{M_i\}_{i\in\Gamma}$ be a direct system of absolutely pure $R$-modules. Then there is a pure exact sequence  $0\rightarrow K\rightarrow \bigoplus\limits_{i\in\Gamma}M_i\rightarrow {\lim\limits_{\longrightarrow _{i\in\Gamma}}}M_i\rightarrow 0$. Note that $\bigoplus\limits_{i\in\Gamma}M_i$ is absolutely pure, so is ${\lim\limits_{\longrightarrow _{i\in\Gamma}}}M_i$ by (2).

	$(3)\Rightarrow (1)$: Suppose  the ring $R$ satisfies that any direct limit of  absolutely pure $R$-modules is $u$-$S$-absolutely pure with respect to $s$. Let $I$ be a finitely generated  ideal of $R$,  $\{M_i\}_{i\in\Gamma}$  a direct system of $R$-modules. Let $\alpha: I\rightarrow \lim\limits_{\longrightarrow _{i\in\Gamma}}M_i$ be a homomorphism. For any $i\in \Gamma$, $E(M_i)$ is the injective envelope of $M_i$. Then $E(M_i)$ is absolutely pure. By (3), we have $s\Ext_R^1(R/I, \lim\limits_{\longrightarrow _{i\in\Gamma}}M_i)=0$. So there exists an $R$-homomorphism $\beta:R\rightarrow \lim\limits_{\longrightarrow _{i\in\Gamma}}E(M_i)$ such that the following  diagram commutes:
	$$\xymatrix@R=20pt@C=25pt{
		0\ar[r]^{}&I\ar[r]^{}\ar[d]_{s\alpha}&  R\ar[r]^{}\ar@{.>}[d]^{\beta} & R/I\ar@{.>}[d]^{}\ar[r]^{} &0\\
		0\ar[r]^{}&{\lim\limits_{\longrightarrow _{i\in\Gamma}}}M_i \ar[r]^{}&{\lim\limits_{\longrightarrow _{i\in\Gamma}}}E(M_i)  \ar[r]^{} &{\lim\limits_{\longrightarrow _{i\in\Gamma}}}E(M_i)/M_i\ar[r]^{} & 0.\\}$$
	Thus, by \cite[Lemma 2.7]{gt}, there exists $j\in \Gamma$, such that $\beta$ can factor through $R\xrightarrow{\beta_j} E(M_j)$. Consider the following commutative diagram:
	$$\xymatrix@R=20pt@C=25pt{
		0\ar[r]^{}&I\ar[r]^{}\ar@{.>}[d]_{s\alpha_j}&  R\ar[r]^{}\ar[d]^{\beta_j} & R/I\ar@{.>}[d]^{}\ar[r]^{} &0\\
		0\ar[r]^{}&M_j \ar[r]^{}&E(M_j)  \ar[r]^{} &E(M_j)/M_j\ar[r]^{} & 0.\\}$$
	Since the composition $I\rightarrow R\rightarrow E(M_j)\rightarrow E(M_j)/M_j$ becomes to be $0$ in the direct limit, we can assume $I\rightarrow R\rightarrow E(M_j)$ can factor through some  $I\xrightarrow{s\alpha_j} M_j$. Thus $s\alpha$ can factor through $M_j$. Consequently, the natural homomorphism  $\lim\limits_{\longrightarrow _{i\in\Gamma}} \Hom_R(I,M_i)\xrightarrow{\phi}  \Hom_R(I, \lim\limits_{\longrightarrow _{i\in\Gamma}}M_i)$ is a $u$-$S$-epimorphism with respect to $s$. Now suppose  $\{M_i\}_{i\in\Gamma}$ is a direct system of finitely presented $R$-modules such that $\lim\limits_{\longrightarrow _{i\in\Gamma}} M_i=I$. Then there exists $f\in \Hom_R(I,M_j)$ with $j\in \Gamma$ such that the identity map  $s\Id_I= \phi(u_j(f))$ where $u_j$ is the natural homomorphism $\Hom_R(I,M_j)\rightarrow \lim\limits_{\longrightarrow _{i\in\Gamma}} \Hom_R(I,M_i)$. That is, we have the following commutative diagram: 	
	$$\xymatrix@R=10pt@C=40pt{
		I\ar[rd]^{f}\ar[dd]_{s\Id_I}&\\
		&M_i\ar[ld]^{\delta_i}\\
		I=\lim\limits_{\longrightarrow _{i\in\Gamma}} M_i&\\}$$
	So, by the proof of \cite[Lemma 2.4]{ZQusproj}, $0\rightarrow \Ker{(\delta_i)} \rightarrow M_i\xrightarrow{\delta_i} \Im(\delta_i)\rightarrow 0$ is a $u$-$S$-split short exact sequence with respect to $s^3$. It follows by the proof of Lemma \ref{dir usfp} that	there is a finitely presented $R$-module $F$ and a $u$-$S$-torsion module $T_1$ (with respect to $s^3$) such that the sequence $0\rightarrow T_1\rightarrow F\rightarrow \Im(\delta_i) \rightarrow0$ exact. Combining the exact sequence $0\rightarrow \Im(\delta_i)\rightarrow I \rightarrow I/\Im(\delta_i) \rightarrow 0$, we obtain a long exact sequence $$0\rightarrow T_1\rightarrow F\rightarrow I \rightarrow I/\Im(\delta_i) \rightarrow 0.$$ Note that $s (I/\Im(\delta_i))=0.$
Consequently,	$I$ is $u$-$S$-finitely presented with respect to $s^3$. Hence $R$ is a $u$-$S$-coherent ring.	
\end{proof}
 \begin{remark}  It is well-known that a ring $R$ is coherent if and only if the class of absolutely pure $R$-modules is (pre)covering (see \cite{DD18,DD19}). It is a natural question that does the following statement hold?
 	
  A ring $R$ is $u$-$S$-coherent if and only if any $R$-module has a $u$-$S$-absolutely pure $u$-$S$-(pre)cover (see Definition \ref{usprecover}).
 \end{remark}

\section{Characterizing $u$-$S$-Noetherian rings by $u$-$S$-injective modules}

 Let $R$ be a ring and $S$ a multiplicative subset of $R$. Recall from \cite{CKQWZ} that an $R$-module $E$ is said to be {\it $u$-$S$-injective} (with respect to $s$) provided that the induced sequence $$0\rightarrow \Hom_R(C,E)\rightarrow \Hom_R(B,E)\rightarrow \Hom_R(A,E)\rightarrow 0$$ is $u$-$S$-exact (with respect to $s$) for any exact sequence $0\rightarrow A\rightarrow B\rightarrow C\rightarrow 0$ of $R$-modules, or equivalently,
 	  $\Ext_R^1(M,E)$ is  $u$-$S$-torsion with respect to $s$ for any  $R$-module $M$ (see \cite[Lemma 2.4]{Bourzri26}).
 		

 \begin{lemma}\label{ussplit-usinj}
 	Let $0\rightarrow A\rightarrow B\rightarrow C\rightarrow 0$ be a $u$-$S$-split  short $u$-$S$-exact sequence. Suppose $B$ is $u$-$S$-injective. Then so are $A$ and $C$.
 \end{lemma}
 \begin{proof}	Let $\xi:0\rightarrow A\xrightarrow{f} B\xrightarrow{g} C\rightarrow 0$ be a $u$-$S$-split  short $u$-$S$-exact sequence with $B$  $u$-$S$-injective.  Let $X$ be a module. We will prove that $\Ext_R^1(X,A)$ is annihilated by some element in $S$.
 	Let  $f':B \rightarrow A$ such that $f'\circ f = s\Id_A$ with $s \in S$.
 	Then $f,f'$ induce the maps $\Ext_R^1(X,f): \Ext_R^1(X,A) \rightarrow \Ext_R^1(X,B)$ and $\Ext_R^1(X,f'): \Ext_R^1(X,B) \rightarrow \Ext_R^1(X,A)$  with  $\Ext_R^1(X,f')\circ \Ext_R^1(X,f)=s\Id_{\Ext_R^1(X,A)}$.
 	As $t\Ext_R^1(X,B) = 0$ for some $t \in S$, we get $st\Ext_R^1(X,A) = 0$. Consequently, $A$ is  $u$-$S$-injective. It follows by \cite[Proposition 4.7]{CKQWZ} that $C$ is  also $u$-$S$-injective.
 \end{proof}

 An $R$-module $E$ is said to be \emph{$u$-$S$-cofree} if $F$ is $u$-$S$-isomorphic to a cofree module. The latter is defined to be a direct product of copies of $R^+$.

\begin{proposition}\label{us-cofree}
Let $R$ be a ring and $S$ a multiplicative subset of $R$. Then an $R$-module
$E$ is $u$-$S$-injective if and only if $E$ is a direct summand of a $u$-$S$-cofree module.
\end{proposition}
 \begin{proof}
 	Let $E$ be a $u$-$S$-injective. Then $E$ can be embeded into a cofree $R$-module $G$. Consider the  $u$-$S$-split exact sequence $0\rightarrow E\rightarrow G\rightarrow G/E\rightarrow0$. Then $G/E$ is  $u$-$S$-injective, and $G$  is $u$-$S$-isomorphic to $E\oplus G/E$ by \cite[Lemma 2.8]{KMOZ24}. So $E\oplus G/E$ is $u$-$S$-cofree. The converse holds obviously.
 \end{proof}

  	


  It is well-known that a ring $R$ is Noetherian if and only if the class of injective modules is precovering. The main aim of this section is to give the uniform $S$-versions of this result.
 Denote by  $u$-$S\mathcal{I}$ the class of all $u$-$S$-injective modules. The author in \cite{Bourzri26} characterized when $u$-$S\mathcal{I}$ is (pre)covering.
  \begin{theorem} \cite[Theorem 2.16]{Bourzri26}
 	Let $R$ be a ring, $S$ be a regular  multiplicative subset of $R$. Consider the following statements:
 	\begin{enumerate}
 		\item  $u$-$S\mathcal{I}$ is closed under direct sums;
 		\item   $u$-$S\mathcal{I}$ is closed under direct limits;
 		\item  $u$-$S\mathcal{I}$ is a precovering class;
 		\item  $u$-$S\mathcal{I}$ is a covering class;
 		\item  $R$ is  $u$-$S$-Noetherian;
 	\end{enumerate}
 	Then $(1)\Leftrightarrow (2)\Leftrightarrow (3)\Leftrightarrow (4)\Rightarrow (5)$.
 \end{theorem}

Subsequently, the author \cite{Bourzri26} proposed the following question:\\
\textbf{Question}  \cite[Question 2.17]{Bourzri26}  Does the implication $(5)\Rightarrow (1)$ in the above theorem hold? If not, is there a way to define a uniform $S$-version of the notion of covers so that we can obtain a uniform $S$-version of Enochs-Teply's characterization of Noetherian rings.

The following example shows that the implication $(5)\Rightarrow (1)$ in \cite[Theorem 2.16]{Bourzri26}   may be wrong even for Noetherian rings.

\begin{example}\label{uf not-extsion}{\rm
		Let $R=\mathbb{Z}$ be the ring of integers, $p$ a prime in $\mathbb{Z}$ and  $S=\{p^n \mid n\geq 0\}$. Then  each $\mathbb{Z}/\langle p^k\rangle$ is  $u$-$S$-torsion, and thus is $u$-$S$-injective. However, $\bigoplus\limits_{k=1}^\infty \mathbb{Z}/\langle p^k\rangle$ is not  $u$-$S$-injective. Indeed, since each $\mathbb{Z}/\langle p^k\rangle$ is a reduced group, so is $\bigoplus\limits_{k=1}^\infty \mathbb{Z}/\langle p^k\rangle$. It follows by \cite[Chapter 9, Corollary 6.7]{FS15} that $\bigoplus\limits_{k=1}^\infty \mathbb{Z}/\langle p^k\rangle$ is a subgroup of	$\Ext^1_{\mathbb{Z}}(\mathbb{Q}/\mathbb{Z},\bigoplus\limits_{k=1}^\infty \mathbb{Z}/\langle p^k\rangle)$. Since  $\bigoplus\limits_{k=1}^\infty \mathbb{Z}/\langle p^k\rangle$ is not $u$-$S$-torsion, so is $\Ext^1_{\mathbb{Z}}(\mathbb{Q}/\mathbb{Z},\bigoplus\limits_{k=1}^\infty \mathbb{Z}/\langle p^k\rangle)$. Hence  the direct sum $\bigoplus\limits_{k=1}^\infty \mathbb{Z}/\langle p^k\rangle$ is not $u$-$S$-injective.}
\end{example}

Next, we will try to give a uniform $S$-version of Enochs-Teply's characterization of Noetherian rings.
 We recall some notions on uniform $S$-versions of precovering and covering from \cite{KMOZ24}.

 \begin{definition}\label{usprecover} \cite{KMOZ24} Let $M$ be an $R$-module and $\A$ a class of $R$-modules.
 	\begin{enumerate}
 		\item An $R$-homomorphism  $f\in \Hom_R(A,M)$ with $A\in\A$ is called an  $\A$ $u$-$S$-precover of $M$ provided that  $\Hom_R(A',f):\Hom_R(A',A)\rightarrow \Hom_R(A',M)$ is a $u$-$S$-epimorphism for any $A'\in \A$.
 		
 		\item An $\A$ $u$-$S$-precover $f\in \Hom_R(A,M)$ of $M$ is called an $\A$ $u$-$S$-cover of $M$, provided that $f=f\circ \alpha$ implies that $\alpha$ is a $u$-$S$-isomorphism for any $\alpha\in \End_R(A)$.
 		
 		\item If every $R$-module has an $\A$  $u$-$S$-precover $($resp., $u$-$S$-cover$)$, then $\A$ is called a $u$-$S$-precovering $($resp., $u$-$S$-covering$)$ class.
 	\end{enumerate}
 \end{definition}

 For $s\in S$, let $s\mathcal{I}$ denote the class of all $R$-modules that are $u$-$S$-injective with respect to $s$.  Now, we have the following new uniform $S$-version of Enochs-Teply's characterization of Noetherian rings.

 \begin{theorem}\label{s-injective-ext}
 		Let $R$ be a ring, $S$ be a regular  multiplicative subset of $R$. Consider the following statements:
 	\begin{enumerate}
 		\item  $R$ is  $u$-$S$-Noetherian;
 		\item  for all $s\in S$, any direct sum  of  $R$-modules in $s\mathcal{I}$ is $u$-$S$-injective;
 		\item  any direct sum  of  injective $R$-modules   is $u$-$S$-injective;
 			\item every $R$-module has a $u$-$S$-injective $u$-$S$-precover.
 	\end{enumerate}
 Then $(1)\Leftrightarrow (2)\Leftrightarrow (3)\Leftarrow (4)$. If, moreover, every $u$-$S$-injective module is $u$-$S$-isomorphic to an injective module, then $(3)\Rightarrow (4)$.
 \end{theorem}

 \begin{proof}
 $(2)\Rightarrow  (3)$ Trivial.  $(3)\Rightarrow  (1)$ It follows by \cite[Theorem 4.10]{CKQWZ}.

 	$(1)\Rightarrow(2)$ Let $R$ be a  $u$-$S$-Noetherian ring. Then  it is also a $u$-$S$-coherent ring by \cite[Proposition 3.10]{Zuscoh-24}. Let $\{E_i\}_{i\in \Lambda}$ be a family of $u$-$S$-injective modules with respect to $s$.  Then for each $i\in \Lambda$, $E_i^+$ is $u$-$S$-flat with respect to some fixed $s_1\in S$ by the proof of \cite[Theorem 4.7]{Zuscoh-24}. Consequently, $\prodi E_i^+\cong (\bigoplus\limits_{i\in \Lambda}E_i)^+$ is $u$-$S$-flat by Theorem \ref{usflatpreenvelope}. Hence, $\bigoplus\limits_{i\in \Lambda}E_i$ is $u$-$S$-injective by \cite[Theorem 2.14]{Bourzri26}.

 $(4)\Rightarrow (3)$: Let $\{E_i\mid i\in\Lambda\}$ be a family of injective $R$-modules. Then, by assumption, $\bigoplus\limits_{i\in \Lambda}E_i$ has a $u$-$S$-injective $u$-$S$-precover $f:E\rightarrow \bigoplus\limits_{i\in \Lambda}E_i$. Hence there exists $s_1\in S$ such that for any $i\in\Lambda$ we have the following commutative diagram:
 $$\xymatrix{
  E_i\ar[r]^{\times s_1}\ar[d]^{g_i} & E_i \ar[d]^{e_i} \\
 	E\ar[r]^{f} &\bigoplus\limits_{i\in \Lambda}E_i,}$$
 Consequently, $s_1\Id_{\bigoplus\limits_{i\in \Lambda}E_i}= f\circ(\bigoplus\limits_{i\in \Lambda} g_i)$. So  $$0\rightarrow \Ker(f)\rightarrow E\rightarrow \bigoplus\limits_{i\in \Lambda}E_i\rightarrow 0$$	is a $u$-$S$-split  short $u$-$S$-exact sequence.  It follows by Lemma \ref{ussplit-usinj} that $\bigoplus\limits_{i\in \Lambda}E_i$ is $u$-$S$-injective.

 Now, assume that every $u$-$S$-injective module is $u$-$S$-isomorphic to an injective module.

 $(3)\Rightarrow (4)$ Let $M$ be an $R$-module. Let $\{E_i\mid i\in\Lambda\}$ be the set of  injective envelopes of all cyclic $R$-modules,  $X_i=\Hom_R(E_i,M)$, and $f_i:E_i^{(X_i)}\rightarrow M$ be the evaluation map $(\phi_g)_{g\in X_i}\mapsto \sum\limits_{g\in X_i}g(\phi_g)$. Then any $R$-homomorphism $E'\rightarrow M$ with $E'$ injective factor through the natural $R$-homomorphism  $\bigoplus\limits_{i\in \Lambda}f_i$.  We claim that  $\bigoplus\limits_{i\in \Lambda}f_i:\bigoplus\limits_{i\in \Lambda}E_i^{(X_i)}\rightarrow M$ is a $u$-$S$-injective $u$-$S$-precover. Indeed, by assumption, $\bigoplus\limits_{i\in \Lambda}E_i^{(X_i)}$ is a $u$-$S$-injective $R$-module.   Let $f:E\rightarrow M$ be an $R$-homomorphism with $E$ $u$-$S$-injective. Then  there is a $u$-$S$-isomorphism $t':E\rightarrow E'$ with $E'$  injective  by assumption. Hence there is $s'\in S$  and a $u$-$S$-isomorphism $t:E'\rightarrow E$ such that $t \circ t' = s' \Id_E$ and $t' \circ t = s' \Id_{E'}$ by \cite[Lemma 1.2]{ZQusproj}. Then there exists an $R$-homomorphism $g:E'\rightarrow \bigoplus\limits_{i\in \Lambda}E_i^{(X_i)}$ such that $\bigoplus\limits_{i\in \Lambda}f_i\circ g=f\circ t$:
 $$\xymatrix{
 	E'\ar[r]^{t}\ar[d]^{g} & E\ar@/^/[l]^{t'}\ar[d]^{f} \\
 	\bigoplus\limits_{i\in \Lambda}E_i^{(X_i)}\ar[r]^{\bigoplus\limits_{i\in \Lambda}f_i} &M.}$$
 Hence, $\bigoplus\limits_{i\in \Lambda}f_i\circ g\circ t'=f\circ t\circ t'=s'f$. Consequently, $\bigoplus\limits_{i\in \Lambda}f_i$ is a $u$-$S$-injecive  $u$-$S$-precover.	
 \end{proof}

\begin{remark} The equivalence of $(1)$ and $(3)$ in Theorem \ref{s-injective-ext}  can be found in \cite[Theorem 4.10]{CKQWZ}.  As to the equivalence of $(1)$ and $(4)$, we don't known whether every $u$-$S$-injective module is $u$-$S$-isomorphic to an injective module in general. It is a natural question whether the following statement holds?
	
A ring $R$ is  $u$-$S$-Noetherian if and only if every $R$-module has a $u$-$S$-injective $u$-$S$-cover.
\end{remark}

\begin{acknowledgement} The authors greatly appreciate the reviewer's comments and revision suggestions on this article.
\end{acknowledgement}

\end{document}